\documentclass[10pt]{amsart}
\usepackage[]{amsmath, amsthm, amsfonts,  scalerel}
\usepackage{graphicx}
\usepackage[all]{xy}
 \usepackage{xcolor}
 \usepackage{tikz-cd}
\usepackage{adjustbox}
\usepackage{mathtools}
\usepackage{hyperref}

\usepackage[
backend=biber,
style=alphabetic
]{biblatex}

\usepackage[left=3cm,right=3cm,top=3cm,bottom=3cm]{geometry}
\usepackage{cleveref}
\usepackage[left=3cm,right=3cm,top=3cm,bottom=3cm]{geometry}
\usepackage{cleveref}

\newcommand {\C} {{\mathbb C}}
\newcommand {\R} {{\mathbb R}}
\newcommand {\Z} {{\mathbb Z}}
\newcommand {\Q} {{\mathbb Q}}

\newcommand {\cH} {{\mathcal H}}

\newcommand {\cE} {{\mathcal E}}

\newcommand {\cM} {{\mathcal M}}

\newcommand {\cY} {{\mathcal Y}}

\newcommand {\cO} {{\mathcal O}}
\newcommand {\D} {\mathbb{D}}

\newcommand{\tX}{\tilde{X}}

\newcommand{\cD}{\mathcal{D}}

\newcommand{\cX}{\mathcal{X}}
\newcommand{\bR}{\textbf{R}}

\newcommand{\Omp}{\underline{\Omega}_{X}^{p}}

\DeclareMathOperator{\im}{im}

\DeclareMathOperator\supp{supp}
\newcommand{\bT}{\begin{tikzcd}}
\newcommand{\eT}{\end{tikzcd}}
\DeclareMathOperator\depth{depth}
\DeclareMathOperator\lcd{lcd}
\DeclareMathOperator\lcdef{lcdef}
\DeclareMathOperator\codim{codim}

\newtheorem{thm}[subsection]{Theorem}
\newtheorem{cor}[subsection]{Corollary}
\newtheorem{lemma}[subsection]{Lemma}
\newtheorem{prop}[subsection]{Proposition}
\newtheorem{defn}[subsection]{Definition}
\newtheorem{rmk}[subsection]{Remark}

\begin{document}
\author{ Scott Hiatt }
\date{\today}
   \address{
 Department of Mathematics\\
  University of Wisconsin-Madison\\
  Madison, WI 53706\\
  U.S.A.}
  \email{shiatt@wisc.edu}
 %\thanks{ To Someone}

 \title{Differential forms on varieties with pre-$k$-rational singularities}

\maketitle

\begin{abstract}
 Let $X$ be a complex algebraic variety. With $\Q$-coefficients, the compactly supported cohomology groups $H^{i}_{c}(X, \Q)$ and the compactly supported intersection cohomology groups $IH^{i}_{c}(X, \Q)$ have mixed Hodge structures. We compare these two mixed Hodge structures for varieties with pre-$k$-rational singularities. We then study various notions of differential forms on varieties with pre-$k$-rational singularities. In particular, we investigate the depth of the complex $\Omp$,  where $\Omp$ is the $p^{th}$-graded piece of the Du Bois complex. 
\end{abstract}

%\tableofcontents

\section*{Introduction}

Following \cite{MuOlPoWi} \cite{JuKiSaYo}  \cite{FrLa}, when a complex algebraic variety $X$ is a local complete intersection, the variety is said to have $k$-Du Bois singularities if the canonical morphism
$$\Omega^{p}_{X} \rightarrow \underline{\Omega}^{p}_{X}$$
is a quasi-isomorphism for $0 \leq p \leq k,$ where $\Omp$ is the $p^{th}$-graded piece of the Du Bois complex. $X$ is said to have $k$-rational singularities if the canonical morphism
$$\Omega^{p}_{X}\rightarrow \D(\underline{\Omega}^{n-p}_{X})[-\dim X]$$
is a quasi-isomorphism for $0 \leq p \leq k$, where $\D = \bR \cH om_{X}(\bullet, \omega_{X}^{\bullet}).$ For general quasi-projective varieties, we will follow the definitions given by \cite{shenVenVo} for higher singularities on algebraic varieties.

\begin{defn}
$X$ has $k$-Du Bois singularities if it is seminormal and
\begin{enumerate}
    \item $\codim_{X}(S) \geq 2k+1$, where $S$ is the singular locus;
    \item $X$ has pre-$k$-Du Bois singularities. That is,
    $$\cH^{i}(\underline{\Omega}^{p}_{X}) = 0 \quad \text{for $i>0$ and $0 \leq p \leq k.$}$$
    \item $\cH^{0}(\Omp)$ is reflexive, for all $p \leq k.$
\end{enumerate}
\end{defn}

\begin{defn}
    $X$ is said to have $k$-rational singularities if it is normal and 
    \begin{enumerate}
    \item $\codim_{X}(S) \geq 2k+2$, where $S$ is the singular locus;
    \item $X$ has pre-$k$-rational singularities. That is,
    $$\cH^{i}(\D(\underline{\Omega}^{n-p}_{X})[-\dim X])= 0 \quad \text{for $i>0$ and $0 \leq p \leq k.$}$$
\end{enumerate}
\end{defn}
Note that when $k = 0$, we receive the commonly known definitions for Du Bois and rational singularities, respectively. It was shown in \cite{shenVenVo} that all the respective definitions agree when $X$ is a local complete intersection. Furthermore, for normal varieties, the following implications were shown
$$\bT \text{$k$-rational singularities} \arrow[r, Rightarrow] \arrow[d, Rightarrow] & \text{$k$-Du Bois singularities} \arrow[d, Rightarrow] \\
\text{pre-$k$-rational singularities} \arrow[r, Rightarrow] & \text{pre-$k$-Du Bois singularities.} \eT$$

%We also have the pure Hodge module $IC^{H}_{X}$ whose de Rham complex is the intersection complex. By \cite{saito}, for $p \in \Z$, we have the complex
%$$I\Omp:= Gr^{F}_{-p}DR(IC^{H}_{X})[p-n] \in D^{b}_{coh}(\cO_{X}).$$
%The relationship between the complexes $\Omp$, $I\Omp$, and $\D(\underline{\Omega}^{n-p}_{X})[-\dim X]$ is given by the natural maps
%$$\Omp \rightarrow I\Omp \rightarrow \D(\underline{\Omega}^{n-p}_{X})[-\dim X].$$

In this paper, we investigate how pre-$k$-rational singularities affect the mixed Hodges structures of $H^{i}_{c}(X, \Q)$ and $IH^{i}_{c}(X, \Q):= H^{i-\dim X}_{c}(X, IC_{X}(\Q))$ and extend some of the results given in \cite{AH2}. When $X$ has rational singularities, it was shown in \cite{AH2} the natural maps
 $$Gr^{0}_{F}H^{i}_{c}(X, \C) \rightarrow Gr^{0}_{F}IH^{i}_{c}(X, \C)  $$
$$Gr^{i}_{F}H^{i}_{c}(X, \C) \rightarrow Gr^{i}_{F}IH^{i}_{c}(X, \C) $$
are isomorphisms for all $i$. In this paper, we prove the following generalization.

\begin{thm}\label{Thm1}
     If $X$ is an irreducible, normal, quasi-projective variety with pre-$k$-rational singularities, then for all $i$, the natural maps
    $$Gr^{p}_{F}H^{i}_{c}(X, \C) \rightarrow Gr^{p}_{F}IH^{i}_{c}(X, \C)$$
    $$Gr^{i-p}_{F}H^{i}_{c}(X, \C) \rightarrow Gr^{i-p}_{F}IH^{i}_{c}(X, \C)$$
   are isomorphisms for $0 \leq p \leq k$. In particular, the natural map of mixed Hodge structures 
   $$H^{i}_{c}(X, \Q) \rightarrow IH^{i}_{c}(X,\Q)$$
   is an isomorphism for $i \leq 2k+1$ and injective for $i = 2k+2.$
\end{thm}

We apply the previous theorem to study the depth of the complex $\Omp$. If $(R,m)$ is the local ring for a closed point $x \in X$ and $M$ is an element of the bounded derived category of $R$-modules, we define 
$$\depth(M) := \min\{i \hspace{.03in} | \hspace{.03in} \cH^{-i}(\bR H om_{R}(M, \omega_{R}^{\bullet}))\neq 0\} = \min \{i \hspace{.03in} | \hspace{.03in} \cH^{i}(\bR \Gamma_{m}(M)) \neq 0 \}.$$ 
The depth of $\Omp$ is given by
$$\depth(\underline{\Omega}^{i}_{X}):= \min \{\depth((\underline{\Omega}^{i}_{X})_{x})| \hspace{.03in} \text{$x \in X$ is any closed point}\}.$$
If $X$ is irreducible and has rational singularities, it was also shown in \cite{AH2},
$$\depth(\Omp) \geq 2 \quad \text{for $0 \leq p \leq \dim(X),$}$$
and the map
$$\cH^{i}(\Omp) \rightarrow  \cH^{i}(I\Omp)$$
is an isomorphism for $i = 0$ and injective for $i =1$ for $0 \leq p \leq \dim(X).$ Here we have
$$I\Omp:= Gr^{F}_{-p}DR(IC^{H}_{X})[p-n] \in D^{b}_{coh}(\cO_{X}),$$
where $IC^{H}_{X}$ is the pure Hodge module on $X$ whose de Rham complex is given by the intersection complex. Note, the relationship between the complexes $\Omp$, $I\Omp$, and $\D(\underline{\Omega}^{n-p}_{X})[-\dim X]$ is given by the natural maps
$$\Omp \rightarrow I\Omp \rightarrow \D(\underline{\Omega}^{n-p}_{X})[-\dim X].$$
In this paper, we prove the following generalization of the result of \cite{AH2}.

\begin{thm}\label{Thm2}
   Assume $X$ is an irreducible, normal, quasi-projective variety and $k \leq \dim X -2$. If $X$ has pre-$k$-rational singularities, then
    $$\depth(\underline{\Omega}^{p}_{X}) \geq k+2  \quad \text{for $0 \leq p \leq \dim(X)$}.$$
    Furthermore, for $0 \leq p \leq \dim(X)$, the map
    $$\cH^{i}(\Omp) \rightarrow \cH^{i}(I\Omp)$$
    is an isomorphism for $i \leq k$ and injective for $i = k+1$. 
\end{thm}

 If $\Delta$ a closed subset of $X$, $$\depth_{\Delta}(\Omp) := \displaystyle \min \{\depth ((\Omp)_{q})| \hspace{.03in} \text{$q \in \Delta$ is any point, not necessarily closed}\}.$$
  In terms of local cohomology, we have the following equivalence, 
 $$\depth_{\Delta}(\Omp) \geq  \ell \quad  \Leftrightarrow \quad \cH^{i}_{\Delta}(\Omp):= \cH^{i}(\bR \underline{\Gamma_{\Delta}}(\Omp)) = 0 \quad \text{for $i \leq \ell -1$.}$$
 If $X$ has rational singularities, it was also shown in \cite{AH2},
$$\depth_{S}(\Omp) \geq 2 \quad \text{for $0 \leq p \leq \dim(X).$}$$
We prove the following generalization.

\begin{thm}
   Assume $X$ is an irreducible, normal, quasi-projective variety and $k \leq \codim_{X}(S) -2$. If $X$ has pre-$k$-rational singularities, then
    $$\depth_{S}(\underline{\Omega}^{p}_{X}) \geq k+2  \quad \text{for $0 \leq p \leq \dim(X)$}.$$
    
\end{thm}

With Theorem \ref{Thm2},  we can obtain an upper bound for the local cohomological dimension of $X$, which matches the upper bound for the local cohomological dimension given by Park and Popa \cite[Corollary 7.5]{ParkPopa}. Recall that for a smooth variety $Y$ such that $X \subseteq Y$ is a closed subvariety, we define the local cohomological dimension of $X$ in $Y$ to be
$$\lcd(Y,X) = \max\{q| \hspace{.05in}\cH^{q}_{X}(\cO_{Y}) \neq 0\}.$$
There is also the local cohomological defect $\lcdef(X)$ of $X$, which is defined as
$$\lcdef(X) := \lcd(Y,X) - \codim_{X}(Y).$$
By the results of Musta\c{t}\u{a} and Popa \cite{MuPo}, we can describe the local cohomological dimension of $X$ in $Y$, or the local cohomological defect, in terms of the depth of $\Omp.$ From Theorem \ref{Thm2} and the results of Musta\c{t}\u{a} and Popa, in the last section of this paper we discuss upper bounds for the local cohomological dimension of $X$ whenever $X$ has pre-$k$-rational singularities.

\subsection*{Acknowledgments} The author would like to thank Sung Gi Park for helpful discussions.

%\begin{cor}
     %Assume $X$ is a normal, quasi-projective variety and $k \leq codim_{X}(S)-2$. If $X$ has pre-$k$-rational singularities, then the natural map
    %$$\cH^{i}(\underline{\Omega}_{X}^{p}) \rightarrow \cH^{i}(I\underline{\Omega}^{p}_{X})$$
    %is an isomorphism for $i \leq k$ and injective for $i = k+1$ for $0 \leq p \leq \dim X.$
%\end{cor}

\section{pre-$k$-rational singularities and mixed Hodge structures}

A variety always means a reduced, separated scheme of finite type over $\C$. However, we will primarily work in the analytic category. So, any variety also has the associated analytic structure.  Recall, that for any complex algebraic variety $X$, we have an object $(\underline{\Omega}^{\bullet}_{X}, F)$ in the filtered derived category (see \cite{dubois}) with the following properties:
\begin{enumerate}
\item If we forget the filtration, then there is an isomorphism $\underline{\Omega}^{\bullet}_{X} \simeq \C_{X}$ in the usual derived category.
\item If $\Omega^{\bullet}_{X}$ is the usual de Rham complex with the ``stupid" filtration, there exists a natural map of filtered complexes
$$(\Omega^{\bullet}_{X},F) \rightarrow (\underline{\Omega}^{\bullet}_{X},F).$$
If $X$ is smooth, this map is an isomorphism.
\item If $\pi_\bullet: X_\bullet \to X$ is a simplicial resolution, then 
$$Gr^{p}_{F}\underline{\Omega}^{\bullet}_{X}[p]:=\underline{\Omega}_X^p \simeq \R \pi_{\bullet} \Omega_{X_\bullet}^p.$$
\end{enumerate}
For the rest of this section, we have $X$ to be an irreducible, normal, quasi-projective variety of dimension $n$ with singular locus $S$.

From the natural map $a_{X}:X \rightarrow \{pt\}$, we have the object $\Q^{H}_{X}[n]:= a_{X}^{*}\Q^{H}_{pt}[n] \in D^{b}MHM(X)$ in the derived category of mixed Hodge modules. We also have the natural quotient map 
$$\Q^{H}_{X}[n] \rightarrow Gr^{W}_{n}\cH^{0}(\Q^{H}_{X}[n])=IC^{H}_{X}$$
in the derived category of mixed Hodge modules, where $IC^{H}_{X}$ is the pure Hodge module on $X$ whose de Rham complex is given by the intersection complex \cite[\S 4]{saito2}. The map $\Q^{H}_{X}[n] \rightarrow IC^{H}_{X}$ is an isomorphism away from the singular set $S.$ Let $\mathcal{K}= Cone(\Q^{H}_{X}[n] \rightarrow IC^{H}_{X})$ with the exact triangle
        \begin{equation}\label{eq1}
            \bT \Q^{H}_{X}[n] \ar[r] & IC^{H}_{X} \ar[r] & \mathcal{K} \ar[r, "+1"] & \hfill. \eT
        \end{equation}
    Apply the functor $Gr^{F}_{-p}DR(\bullet)$ to the previous  exact triangle, and we obtain the exact triangle
    \begin{equation}\label{eq2}
        \bT \underline{\Omega}^{p}_{X}[n-p] \ar[r] &  Gr^{F}_{-p}DR(IC^{H}_{X}) \ar[r] & Gr^{F}_{-p}DR(\mathcal{K}) \ar[r, "+1"] & \hfill. \eT
    \end{equation}
    The quasi-isomorphism $Gr^{F}_{-p}DR(\Q^{H}_X[n])\simeq \underline{\Omega}^{p}_{X}[n-p]$ is due to Saito \cite{saito5}. To simplify notation, we define
    $$I\Omp := Gr^{F}_{-p}DR(IC^{H}_{X})[p-n]$$
    $$\Omega^{p}(\mathcal{K}):= Gr^{F}_{-p}DR(\mathcal{K})[p-n].$$
    So, the exact triangle \ref{eq2} is changed to
    \begin{equation}\label{phimap}
        \bT \underline{\Omega}^{p}_{X} \ar[r, "\phi_{p}"] &  I\Omp \ar[r] & \Omega^{p}(\mathcal{K}) \ar[r, "+1"] & \hfill. \eT
    \end{equation}
    Note that for any resolution of singularities $f: \tX \rightarrow X$, $\cH^{0}(I\Omp) \cong f_{*}\Omega^{p}_{\tX}$  (see \cite{ks} for details).

In \cite{AH2}, for $\cM \in D^{b}MHM(X)$ in the derived category of mixed Hodge modules, the de Rham filtration for $H^{i}(X, DR(\cM))$ and $H^{i}_{c}(X, DR(\cM))$ was constructed, which was constructed by embedding $X$ in a smooth projective variety $\cX$. For the reader's convenience, we review this construction. Choose an open embedding $\iota: X \hookrightarrow X'$ with $X'$ projective, and let $\rho: X' \hookrightarrow \cX$ be any closed embedding, where $\cX$ is a smooth projective variety. The composition of these maps will be denoted as $i_{X}: X \hookrightarrow \cX.$  There exists an open subset $\cY \subset \cX$ such that $i:X \hookrightarrow \cY$ is a closed embedding, and there is a commutative diagram
   $$\adjustbox{scale=1.25}{\bT X \arrow[r, "\iota"] \arrow[d, "i"] \arrow[dr, "i_{X}"] & X' \arrow[d, "\rho"] \\
         \cY \arrow[r, "j"] & \cX. \eT}$$ 
     The Hodge filtration is as follows,
     $$F^{p}H^{i}(X, DR(\cM)) := \im\bigg[H^{i}(\cX,  F_{-p}(DR(i_{X+}\cM))) \rightarrow H^{i}(X, DR(\cM)) \bigg]$$

    $$F^{p}H^{i}_{c}(X, DR(\cM)) := \im\bigg[H^{i}(\cX,  F_{-p}DR(i_{X!}\cM)) \rightarrow H^{i}_{c}(X, DR(\cM)) \bigg].$$
    We define the de Rham filtration in the following way,

    $$^{d}F^{p}H^{i}(X, DR(\cM)) := \im\bigg[H^{i}(\cY, F_{-p}DR(i_{+}\cM)) \rightarrow H^{i}(X, DR(\cM)) \bigg]$$

     $$^{d}F^{p}H^{i}_{c}(X, DR(\cM)) := \im\bigg[H^{i}_{c}(\cY, F_{-p}DR(i_{+}\cM)) \rightarrow H^{i}_{c}(X, DR(\cM)) \bigg].$$

     \begin{rmk}
        As stated in \cite{AH2}, we expect the de Rham filtration to be independent of the embedding and well-defined. However, the results of this paper are independent of the choice of embedding.
     \end{rmk}

     \begin{prop}\label{CompProp}\cite[Prop. 2.4]{AH2}
    For every $p \in \Z$, there exists  natural injective maps $$F^{p}H^{i}(X, DR(\cM)) \hookrightarrow \hspace{.01in} ^{d}F^{p}H^{i}(X, DR(\cM))$$
    $$^{d}F^{p}H^{i}_{c}(X, DR(\cM)) \hookrightarrow F^{p}H^{i}_{c}(X, DR(\cM)).$$
\end{prop}

The Hodge filtration is strict, so the corresponding spectral sequence degenerates at $E_{1}$. However, the de Rham filtration may not be strict, and the associated  sequences,
\begin{equation}\label{eq:deRham}
\begin{split}
   E^{p, i -p}_{1} = H^{i}(\cY, Gr^{F}_{-p}DR(i_{+}\cM)) =H^{i}(X, Gr^{F}_{-p}DR(\cM))& \Rightarrow \hspace{.01in}^{d}Gr_{F}^{p}H^{i}(X, DR(\cM))\\
   E^{p, i -p}_{1} = H^{i}_{c}(\cY, Gr^{F}_{-p}DR(i_{+}\cM)) =H^{i}_{c}(X, Gr^{F}_{-p}DR(\cM)) &\Rightarrow \hspace{.01in}^{d}Gr_{F}^{p}H^{i}_{c}(X, DR(\cM))
\end{split}
\end{equation}
need not degenerate. Considering $\mathcal{K} \in D^{b}MHM(X)$ in the exact sequence \ref{eq1}, we have the following natural maps that correspond to the Hodge and de Rham filtration.
$$\underbrace{H^{i}_{c}(X, \Omega^{p}(\mathcal{K}))}_{ E^{p,i-n}_{1} \text{de Rham filtration}} \rightarrow \underbrace{Gr^{p}_{F}H^{i+p-n}_{c}(X, DR(\mathcal{K}))}_{E^{p,i-n}_{1} \text{Hodge filtration}} \rightarrow \underbrace{H^{i}(X, \Omega^{p}(\mathcal{K}))}_{ E^{p,i-n}_{1} \text{de Rham filtration}}$$

\begin{lemma}\label{Surj. Lemma}
    The map $^{d}Gr^{0}_{F}H^{i}_{c}(X, DR(\mathcal{K})) \rightarrow Gr^{0}_{F}H^{i}_{c}(X,DR(\mathcal{K}))$ is surjective for all $i \in \Z.$
\end{lemma}

\begin{proof}
    Consider the exact triangles \ref{eq1} and \ref{eq2}. We have
    $$Gr^{F}_{p}DR(\Q^{H}_{X}[n]) = 0 \quad \text{for $p>0$.}$$
    $$Gr^{F}_{p}DR(IC^{H}_{X}) = 0 \quad \text{for $p>0$.}$$
    Therefore, $Gr^{F}_{p}DR(\mathcal{K}) = 0$ for $p >0$ and we must have
    $$^{d}F^{0}H^{i}_{c}(X, DR(\mathcal{K})) = H^{i}_{c}(X, DR(\mathcal{K})) \quad \text{for $i \in \Z$.}$$ By Proposition \ref{CompProp}, we must also have $F^{0}H^{i}_{c}(X,DR(\mathcal{K})) = H^{i}_{c}(X,DR(\mathcal{K})).$ Considering the commutative diagram below,
    $$\bT ^{d}F^{0}H^{i}_{c}(X, DR(\mathcal{K})) = H^{i}_{c}(X, DR(\mathcal{K})) \ar[d] \ar[r] & ^{d}Gr^{0}_{F}H_{c}^{i}(X, DR(\mathcal{K})) \ar[d] \ar[r] & 0 \\
    F^{0}H^{i}_{c}(X, DR(\mathcal{K})) = H^{i}_{c}(X, DR(\mathcal{K})) \ar[r] & Gr^{0}_{F}H^{i}_{c}(X, DR(\mathcal{K}))  \ar[r] & 0 \eT$$
    we must have the map $^{d}Gr^{0}_{F}H^{i}_{c}(X, DR(\mathcal{K})) \rightarrow Gr^{0}_{F}H^{i}_{c}(X,DR(\mathcal{K}))$ to be surjective. 
\end{proof}

\begin{thm}\label{HodgeThm}
     If $X$ has pre-$k$-rational singularities, then the natural maps
    $$Gr^{p}_{F}H^{i}_{c}(X, \C) \rightarrow Gr^{p}_{F}IH^{i}_{c}(X, \C)$$
    $$Gr^{i-p}_{F}H^{i}_{c}(X, \C) \rightarrow Gr^{i-p}_{F}IH^{i}_{c}(X, \C)$$
   are isomorphisms for $0 \leq p \leq k.$ In particular, the natural map of mixed Hodge structures 
   $$H^{i}_{c}(X, \Q) \rightarrow IH^{i}_{c}(X,\Q)$$
   is an isomorphism for $i \leq 2k+1$ and injective for $i = 2k+2.$
\end{thm}

\begin{proof}
   
    By the exact triangle \ref{eq1}, there is a long exact sequence
    $$\cdots \rightarrow H^{i-n-1}_{c}(X, DR(\mathcal{K})) \rightarrow  H^{i}_{c}(X, \C) \rightarrow IH^{i}_{c}(X,\C) \rightarrow H^{i-n}_{c}(X, DR(\mathcal{K})) \rightarrow \cdots. $$
      If $X$ has pre-$k$-rational singularities, then by \cite[\S 9]{PopaShenVo}, the map $\phi_{p}$ in the triangle \ref{phimap} is a quasi-isomorphism for $p\leq k$. Therefore, we can conclude $\Omega^{p}(\mathcal{K}) \simeq 0$ for $p \leq k.$ For the de Rham spectral sequence, $E^{p,i-p-n}_{1}= H^{i-p}_{c}(X, \Omega^{p}(\mathcal{K})) =0$ for $p \leq k$ and $i \in \Z$. Thus, 
      $$^{d}Gr^{p}_{F}H^{i-n}_{c}(X, DR(\mathcal{K})) = 0 \quad \text{for $p \leq k$ and $i \in \Z$.}$$
      By Lemma \ref{Surj. Lemma} and induction on $k$, we have
    $$Gr^{p}_{F}H^{i-n}_{c}(X,DR(\mathcal{K})) = 0 \quad \text{for $p \leq k$ and $i \in \Z.$}$$
    By applying the functor $Gr^{W}_{w}(\bullet)$, when $p \leq k$,
    $$Gr^{p}_{F}Gr^{W}_{w}H^{i-n}_{c}(X,DR(\mathcal{K})) =  0 \quad \text{for any $i,w \in \Z.$} $$
    Note that $Gr^{W}_{w}H^{i-n}_{c}(X, DR(\mathcal{K})) = 0$ unless $ w \leq i$ \cite[Prop. 2.10]{AH2}. Therefore, for $p \leq k$,
    $$Gr^{w-i+p}_{F}Gr^{W}_{w}H^{i-n}_{c}(X,DR(\mathcal{K}))= 0 \quad \text{for any $i,w \in \Z.$}$$
    By Hodge symmetry,
    $$Gr^{i-p}_{F}Gr^{W}_{w}H^{i-n}_{c}(X,DR(\mathcal{K}))= 0 \quad \text{for any $i,w \in \Z.$}$$
     Thus,
      $$Gr^{i-p}_{F}H^{i-n}_{c}(X,DR(\mathcal{K})) = 0 \quad \text{for $p \leq k$ and $i \in \Z.$}$$
      
     For the last statement, it suffices to show $Gr^{W}_{w}H^{i-n}_{c}(X, DR(\mathcal{K}))) = 0$ for $i \leq 2k+1$ and $w \in \Z.$ Since $Gr^{W}_{w}H^{i-n}_{c}(X, DR(\mathcal{K}))) = 0$ for $w > i$, it suffices to consider the cases when $w \leq i \leq 2k+1$. Also, since $Gr^{p}_{F}H^{i-n}_{c}(X, DR(\mathcal{K}))) = 0$ for $p < 0,$ it suffices to consider the cases when $0 \leq w \leq i \leq 2k+1$.   Since $0\leq w \leq i\leq 2k+1,$ we have the following identities,
    $$Gr^{W}_{w}H^{i-n}_{c}(X, DR(\mathcal{K})) = \bigoplus_{0 \leq p \leq w} Gr^{p}_{F}Gr^{W}_{w}H^{i-n}_{c}(DR(\mathcal{K}))= \bigoplus_{0 \leq p \leq 2k+1} Gr^{p}_{F}Gr^{W}_{w}H^{i-n}_{c}(DR(\mathcal{K}))$$
    $$=\underbrace{\bigoplus_{0 \leq p \leq k}Gr^{p}_{F}Gr^{W}_{w}H^{i-n}_{c}(DR(\mathcal{K})) \oplus \bigoplus_{0 \leq p \leq k}Gr^{2k+1-p}_{F}Gr^{W}_{w}H^{i-n}_{c}(DR(\mathcal{K}))}_\text{$=0$ by the previous calculation}.$$
   
\end{proof}

\begin{rmk}\label{Rmk1}
The assumption of $X$ having pre-$k$-rational singularities in the previous theorem can be replaced with the weaker condition that $\phi_{p}:\Omp \rightarrow I\Omp$ is a quasi-isomorphism for $0 \leq p \leq k.$  This condition is known as $(*)_{k}$ in \cite{ParkPopa} or $HRH(X)\geq k$ in \cite{DirksOlaoRay}. 
    
    Also, the first natural map in Theorem \ref{HodgeThm},
    $$Gr^{p}_{F}H^{i}_{c}(X, \C) \rightarrow Gr^{p}_{F}IH^{i}_{c}(X, \C)$$ 
    was also shown to be an isomorphism  for $0 \leq p \leq k$ under the $HRH(X)\geq k$ condition by \cite{DirksOlaoRay}.
\end{rmk}

If $X$ is projective and has pre-$k$-rational singularities (or $(*)_{k}$ singularities), the theorem above gives us
$$H^{p,i-p}(X):=H^{i-p}(X, \Omp) \cong Gr^{p}_{F}IH^{i}(X,\C)=:IH^{p,i-p}(X) \quad \text{for $p \leq k$ and $i \in \Z.$}$$
$$H^{i-p,p}(X):=H^{p}(X, \underline{\Omega}^{i-p}_{X}) \cong Gr^{i-p}_{F}IH^{i}(X,\C)=:IH^{i-p,p}(X) \quad \text{for $p \leq k$ and $i \in \Z.$}$$
This was first shown by Park and Popa \cite[Thm. 7.1]{ParkPopa}. In fact, they show more in terms of the symmetry of the Hodge-Du Bois numbers using the $(*)_{k}$ condition. In the case when $X$ has pre-$k$-rational singularities, it was first shown by \cite[Cor. 4.1]{shenVenVo} that we have
$$\dim H^{p,i}(X) = \dim H^{i,p}(X) = \dim H^{n-p,i-n}(X) \quad \text{for $p \leq k$ and $0 \leq i \leq n.$}$$
In particular, if $X$ is a projective variety, then $IH^{i}_{c}(X,\Q) = IH^{i}(X,\Q)$ has a pure Hodge structure of weight $i$.  

\begin{cor}\cite{ParkPopa}
    If $X$ is a projective variety and has pre-$k$-rational singularities (or $(*)_{k}$ singularities), then $H^{i}(X, \Q)$ has a pure Hodge structure of weight $i$ for $i \leq 2k+2$.
\end{cor}

\section{pre-$k$-rational singularities and differential forms} 
We keep the same assumptions and notation as in the previous section. We have $X$ to be an irreducible, normal, quasi-projective variety of dimension $n$ with singular locus $S$. Also, $k$ will be a positive integer such that $ k \leq n -2.$

\begin{lemma}\label{depthlemma}
       If $I\Omp \simeq \cH^{0}(I\Omp)$ for $p \leq k$, then
     $$\cH^{0}(F_{-k-1}DR(IC^{H}_{X}))=\cH^{n-k-1}(F_{-k-1}DR(IC^{H}_{X}[k+1 -n])) = 0 \quad \text{and} \quad \cH^{n-k-1}(I\underline{\Omega}^{k+1}_{X})=0.$$ 
\end{lemma}

\begin{proof}
   We will proceed by induction on $k$. When $k = 0$, the lemma holds by \cite[Lemma 6.8]{ks}. Let $k \geq 1$. Assume that if $I\Omp \simeq \cH^{0}(I\Omp)$ for $p \leq k-1$, then the lemma holds. Suppose that $I\Omp \simeq \cH^{0}(I\Omp)$ for $p \leq k$.  We need to show
     $$\cH^{n-k-1}(F_{-k-1}DR(IC^{H}_{X}[k+1-n])) = 0 \quad \text{and} \quad \cH^{n-k-1}(I\underline{\Omega}^{k+1}_{X})=0.$$
    Consider the short exact sequences
    $$0 \rightarrow F_{-k-1}(IC^{H}_{X}[k-n]) \rightarrow F_{-k}(IC^{H}_{X}[k-n]) \rightarrow I\underline{\Omega}^{k}_{X} \rightarrow 0$$
    $$0 \rightarrow F_{-k-2}(IC^{H}_{X}[k+1-n]) \rightarrow F_{-k-1}(IC^{H}_{X}[k+1-n]) \rightarrow I\underline{\Omega}^{k+1}_{X} \rightarrow 0.$$
    If we apply cohomology to the first exact sequence, then
    $$\underbrace{\cH^{n-k-1}(I\underline{\Omega}^{k}_{X})}_\text{$=0$ by assumption} \rightarrow \cH^{n-k}( F_{-k-1}(IC^{H}_{X}[k-n])) \rightarrow \underbrace{\cH^{n-k}(F_{-k}(IC^{H}_{X}[k-n]))}_\text{$=0$ by induction hypothesis}.$$
    If we apply cohomology to the second exact sequence, then
    $$ \underbrace{\cH^{n-k-1}( F_{-k-1}(IC^{H}_{X}[k+1-n]))}_\text{$=0$ by the above calculation} \rightarrow \cH^{n-k-1}(I\underline{\Omega}^{k+1}_{X}) \rightarrow \underbrace{\cH^{n-k}( F_{-k-2}(IC^{H}_{X}[k+1-n]))}_\text{$=0$ by definition of the de Rham complex}.$$
\end{proof}

\begin{prop}\label{depthProp}
 If $I\Omp \simeq \cH^{0}(I\Omp)$ for $p \leq k$, then $\depth(I\underline{\Omega}^{p}_{X}) \geq k+2$ for $0 \leq p \leq n.$

\end{prop}

\begin{proof}
 For a closed point $x \in X$, we need to show
    $$\cH^{i}(\bR \underline{\Gamma_{x}}(I\underline{\Omega}^{p}_{X})) = 0 \quad \text{for $i \leq k+1.$}$$
     By local duality,
$$\cH^{i}(\bR \underline{\Gamma_{x}}(I\Omp)_{x}) \cong Hom(\cH^{-i}(\bR \cH om_{\cO_{X}}( I\Omp, \omega^{\bullet}_{X}))_{x}, I_{x}),$$
where $I_{x}$ is the injective hull of the residue field at $x \in X$. Since $IC^{H}_{X}$ is a pure (polarizable) Hodge module of weight $n$, there is a quasi-isomorphism 
$$\bR \cH om_{\cO_{X}}(I\Omp, \omega^{\bullet}_{X})) \cong I \underline{\Omega}^{n-p}_{X}[n].$$
So, we have an isomorphism
$$\cH^{i}(\bR \underline{\Gamma_{x}}(I\Omp)_{x}) \cong Hom(\cH^{n -i}(I\underline{\Omega}^{n-p}_{X})_{x} , I_{x}).$$
If $(n -i) + (n-p) >n$, then $\cH^{n -i}(I\underline{\Omega}^{n-p}_{X}) = 0$. So, it suffices to consider the case when $(n-i) + (n-p) \leq n.$ Equivalently, it suffices to consider when $n-p\leq i.$ First, consider when  $n-p \leq i \leq k \leq n-2$. Then, by our assumption,
$$\cH^{n-i-1}I\underline{\Omega}^{n-p}_{X} \cong \cH^{n-i-1}(\cH^{0}(I\underline{\Omega}^{n-p}_{X})) =0.$$
Therefore,
$$\cH^{i}(\bR \underline{\Gamma_{x}}(I\Omp)_{x}) \cong 0 \quad \text{for $n-p \leq k$ and $i \leq k+1.$}$$
The final case is when $n-p = i = k+1$. In this case,
$$\cH^{n-i}(I\underline{\Omega}^{n-p}_{X}) \cong \cH^{n-k-1}(I\underline{\Omega}^{k+1}_{X}).$$
Now apply the previous lemma.

\end{proof}

\begin{thm}\label{depthThm2}
     If $X$ has pre-$k$-rational singularities, then
    $$\depth(\underline{\Omega}^{p}_{X}) \geq k+2  \quad \text{for $0 \leq p \leq n.$}$$
    Furthermore, for $0 \leq p \leq n$, the map
    $$\cH^{i}(\Omp) \rightarrow \cH^{i}(I\Omp)$$
    is an isomorphism for $i \leq k$ and injective for $i = k+1$. 
\end{thm}
\begin{proof}
    The proof is by induction on $ k$. The case when $k = 0$ was shown in \cite{AH2}. Indeed, whenever $X$ has rational singularities, $\depth(\Omp) \geq 2$ for $0 \leq p \leq n$ \cite{AH2}. Furthermore, we have $\depth_{S}(\Omp) \geq 2$ for $0 \leq p \leq n$. This is equivalent to $\cH^{i}_{S}(\Omp):=\cH^{i}(\bR \underline{\Gamma_{S}}(\Omp)) = 0$ for $i \leq 1.$ If we apply the functor $\bR \underline{\Gamma_{S}}$ to the exact triangle
     $$
        \bT \underline{\Omega}^{p}_{X} \ar[r, "\phi_{p}"] &  I\Omp \ar[r] & \Omega^{p}(\mathcal{K}) \ar[r, "+1"] & \hfill, \eT
   $$
    and take cohomology, we obtain the long exact sequence
    $$0 \rightarrow\cH^{0}_{S}(\Omp) \rightarrow \cH^{0}_{S}(I\Omp) \rightarrow \cH^{0}_{S}(\Omega^{p}(\mathcal{K})) \rightarrow \cH^{1}_{S}(\Omp) \rightarrow \cdots.$$
    We have $\cH^{0}_{S}(I\Omp) = \cH^{0}_{S}(\cH^{0}(I\Omp)) = 0$ and $\cH^{0}_{S}(\Omp) = \cH^{1}_{S}(\Omp) = 0.$ Since $\Omega^{p}(\mathcal{K)}$ is supported on $S$,
     $$\cH^{0}_{S}(\Omega^{p}(\mathcal{K})) =  \cH^{0}(\Omega^{p}(\mathcal{K})) = 0.$$
    So, $\cH^{0}(\Omp) \cong \cH^{0}(I\Omp)$, and the map $\cH^{1}(\Omp) \rightarrow \cH^1(I\Omp)$ is injective.
    
    Since the problem is local, we may assume $X$ is affine. Let $k \geq 1$ and assume the statement holds for a variety with pre-$(k-1)$-rational singularities. Suppose $X$ has pre-$k$-rational singularities. We have the long exact sequence
    $$
        \cdots \rightarrow \cH^{i}(\Omp) \rightarrow  \cH^{i}(I\Omp) \rightarrow \cH^{i}(\Omega^{p}(\mathcal{K})) \rightarrow \cH^{i+1}(\Omp) \rightarrow \cdots.
  $$
    If $x \in X$ is a closed point, we also have the long exact sequence of local cohomology
    $$\cdots \rightarrow \cH^{i}_{x}(\Omp) \rightarrow  \cH^{i}_{x}(I\Omp) \rightarrow \cH^{i}_{x}(\Omega^{p}(\mathcal{K})) \rightarrow \cH^{i+1}_{x}(\Omp) \rightarrow \cdots. $$
    Recall, by the results of \cite{shenVenVo} and \cite{PopaShenVo}, if $X$ has pre-$k$-rational singularities, then the maps
    $$\cH^{0}(\Omp) \rightarrow \Omp \rightarrow I\Omp \rightarrow \D(\underline{\Omega}^{n-p}_{X})$$
    are quasi-isomorphisms for $0 \leq p \leq k$. Furthermore, since $X$ has rational singularities, $\cH^{0}(\Omp) \cong \cH^{0}(I\Omp)$ \cite[Thm. 1.4, Cor. 1.12]{ks} (a proof is also given in \cite{AH2}). By the induction hypothesis, and Proposition \ref{depthProp}, $\cH^{i}_{x}(\Omp) = \cH^{i}_{x}(I \Omp) = 0$ for $i \leq k.$ We also have $\cH^{k+1}_{x}(I\Omp) = 0$. Therefore,
    $$\cH^{k}_{x}(\Omega^{p}(\mathcal{K})) \cong \cH^{k+1}_{x}(\Omp).$$
    So it suffices to show $\cH^{k}_{x}(\Omega^{p}(\mathcal{K}))=0$. Also from the induction hypothesis, $\cH^{i}(\Omega^{p}(\mathcal{K}))= 0$ if  $i\leq k-1.$ Thus,
$$\cH^{k}_{x}(\Omega^{p}(\mathcal{K})) \cong 
\underline{\Gamma_{x}}(\cH^{k}(\Omega^{p}(\mathcal{K}))).$$
    By Theorem \ref{HodgeThm}, $Gr^{p}_{F}H^{k +p-n}_{c}(X, DR(\mathcal{K})) = 0$ for $p \in \Z$. Recall that there are natural maps
$$H^{k}_{c}(X, \Omega^{p}(\mathcal{K})) \rightarrow Gr^{p}_{F}H^{k+p-n}_{c}(X, DR(\mathcal{K})) \rightarrow H^{k}(X, \Omega^{p}(\mathcal{K})).$$
By the induction hypothesis,
$$H^{k}_{c}(X, \Omega^{p}(\mathcal{K})) \cong \Gamma_{c}(X, \cH^{k}(\Omega^{p}(\mathcal{K}))) \quad \text{and} \quad H^{k}(X, \Omega^{p}(\mathcal{K})) \cong \Gamma(X, \cH^{k}(\Omega^{p}(\mathcal{K}))).$$
As a consequence, we have a commutative diagram 
$$ \bT \Gamma_{c}(X, \cH^{k}(\Omega^{p}(\mathcal{K}))) \arrow[r] \arrow[hook, dr] & Gr^{p}_{F}H^{k+p-n}_{c}(X, DR(\mathcal{K})) \arrow[d] \\
& \Gamma(X, \cH^{k}(\Omega^{p}(\mathcal{K}))). \eT$$
So the map
$$\Gamma_{c}(X, \cH^{k}(\Omega^{p}(\mathcal{K}))) \rightarrow Gr^{p}_{F}H^{k+p-n}_{c}(X, DR(\mathcal{K}))$$
must be injective, and we obtain
$$\Gamma(X, \underline{\Gamma_{x}}(\cH^{k}(\Omega^{p}(\mathcal{K}))) = \Gamma_{c}(X, \underline{\Gamma_{x}}(\cH^{k}(\Omega^{p}(\mathcal{K}))) \subseteq Gr^{p}_{F}H^{k+p-n}_{c}(X, DR(\mathcal{K}))= 0.$$
Since $X$ is affine, we can conclude 
$$\cH^{k +1}_{x}(\underline{\Omega}^{p}_{X}) \cong \cH^{k}_{x}(\Omega^{p}(\mathcal{K})) \cong \underline{\Gamma_{x}}(\cH^{k}(\Omega^{p}(\mathcal{K}))= 0.$$

To finish the proof, we need to show $\cH^{k}(\Omega^{p}(\mathcal{K})) = 0$ for $0 \leq p \leq n.$ It suffices to show the  $\dim \supp{\cH^{k}(\Omega^{p}(\mathcal{K}))} = 0$ because we showed 
$$ \underline{\Gamma_{x}}(\cH^{k}(\Omega^{p}(\mathcal{K}))) = 0 \quad \text{for all closed points $x \in X.$}$$
The proof will be by induction on $\dim X = n.$ If $\dim X = 2$, then $X$ has isolated singularities and $\dim \supp{\cH^{k}(\Omega^{p}(\mathcal{K}))} = 0$. Let $n \geq 3$ and assume the statement holds for a variety with dimension less than or equal to  $n-1$. Suppose $\dim X =n$ and $X$ has pre-$k$-rational singularities. For a sufficiently general hyperplane section $L$, $L$ is normal and also has pre-$k$-rational singularities \cite{shenVenVo}. Let
$$\mathcal{K}_{L}:= Cone(\Q^{H}_{L}[n-1] \rightarrow IC^{H}_{L}).$$
By \cite[Lemma 6.6]{ParkPopa}, we have the exact triangle
$$\bT \Omega^{p-1}(\mathcal{K}_{L})) \otimes \cO_{L}(-L) \arrow[r] & \Omega^{p}(\mathcal{K}) \otimes \cO_{L} \arrow[r] & \Omega^{p}(\mathcal{K}_{L}) \arrow[r,"+1"] & \hfill \eT$$
and the long exact sequence
$$\cdots \rightarrow \cH^{i}(\Omega^{p-1}(\mathcal{K}_{L}))) \otimes \cO_{L}(-L) \rightarrow \cH^{i}( \Omega^{p}(\mathcal{K})) \otimes \cO_{L} \rightarrow \cH^{i}(\Omega^{p}(\mathcal{K}_{L})) \rightarrow  \cdots$$
By the induction hypothesis, $\cH^{k}(\Omega^{p}(\mathcal{K}_{L})) = 0$ for all values of $p.$ Therefore, by the long exact sequence above, $\cH^{k}(\Omega^{p}(\mathcal{K})) \otimes \cO_{L} = 0$. The natural map
    $$\cH^{k}(\Omega^{p}(\mathcal{K})) \otimes \cO_{X}(-L) \rightarrow  \cH^{k}(\Omega^{p}(\mathcal{K}))$$
    must be surjective. Let $B = \supp{\cH^{k}(\Omega^{p}(\mathcal{K}))})$ and suppose that $\dim B \geq 1$. We may choose $L$ such that $B \cap L \neq \emptyset$. If we localize at any closed point $x \in B \cap L$, there is a surjective map
    $$\cH^{k}(\Omega^{p}(\mathcal{K}))_{x} \otimes \cO_{X}(-L)_{x} \rightarrow  \cH^{k}(\Omega^{p}(\mathcal{K}))_{x}$$
    Since $\cH^{k}(\Omega^{p}(\mathcal{K}))_{x}$ is a coherent, we can conclude that $\cH^{k}(\Omega^{p}(\mathcal{K}))_x = 0$  by applying Nakayama's lemma. This contradicts the fact that $x \in B$. Hence $\dim B =0$.
\end{proof}

\begin{rmk}
A second proof of the first statement of Theorem \ref{depthThm2} has been given by Burke \cite{burke}.
\end{rmk}

\begin{rmk}\label{RHMrmk}
    If $X$ has pre-$k$-rational singularities, it was shown by \cite[Prop. 9.4]{PopaShenVo} that the map
    $$\underline{\Omega}_{X}^{p} \rightarrow I\underline{\Omega}^{p}_{X} $$
     is a quasi-isomorphism for $p \leq k$. From the theorem above, if $X$ has pre-$k$-rational singularities, then the natural map
    $$\cH^{i}(\underline{\Omega}_{X}^{p}) \rightarrow \cH^{i}(I\underline{\Omega}^{p}_{X})$$
    is an isomorphism for $i \leq k$ and injective for $i = k+1$ for all values of $p.$ In particular, we see that the map
    $$\underline{\Omega}_{X}^{p} \rightarrow I\underline{\Omega}^{p}_{X}$$
    is a quasi-isomorphism for $n-k-1 \leq p \leq n.$ This holds because $\cH^{i}(\Omp) = \cH^{i}(I\Omp)=0$ for $i + p >n$ and the map $\cH^{n-p}(\Omp) \rightarrow \cH^{n-p}(I\Omp)$ is surjective. Thus, the map
     $$\underline{\Omega}_{X}^{p} \rightarrow I\underline{\Omega}^{p}_{X}$$
    is a quasi-isomorphism for all values of $p$ whenever $X$ has pre-$k$-rational singularities and $n \leq 2k+2.$
    
    Park and Popa \cite[Prop. 7.4]{ParkPopa} gave a stronger result. They prove that $$\underline{\Omega}_{X}^{p} \rightarrow I\underline{\Omega}^{p}_{X}$$
    is a quasi-isomorphism for $n-k-1 \leq p \leq n$ whenever $X$ satisfies the $(*)_{k}$ (or $HRH(X) \geq k$) condition. 
\end{rmk}

\begin{cor}\label{*k+1}
    If $X$ has pre-$k$-rational singularities and the map $$\cH^{i}(\underline{\Omega}^{k+1}_{X}) \rightarrow \cH^{i}(I\underline{\Omega}^{k+1}_{X})$$
    is an isomorphism for $ k+1<i$ and surjective for $i = k+1,$ then $X$ satisfies the $(*)_{k+1}$ condition.
\end{cor}

\begin{lemma}
    If $k \leq \codim_{X}(S) -2$, then  $\cH^{i}(\bR \underline{\Gamma_{S}}(I\Omp))$ is coherent for $i \leq k+1$
\end{lemma}

\begin{proof}
    We have $\cH^{0}(I\Omp) \cong f_{*}\Omega^{p}_{\tX}$. Thus, we have an exact triangle
    $$\bT f_{*}\Omega^{p}_{\tX} \ar[r] & I \Omp \ar[r] & \tau^{\geq 1} I\Omp \ar[r, "+1"] & \hfill. \eT$$
    If we apply the functor $\bR \underline{\Gamma_{S}}$, we have the exact triangle
    $$\bT \bR \underline{\Gamma_{S}}(f_{*}\Omega^{p}_{\tX}) \ar[r] & \bR \underline{\Gamma_{S}}(I \Omp) \ar[r] & \bR \underline{\Gamma_{S}}(\tau^{\geq 1} I\Omp) \ar[r, "+1"] & \hfill. \eT$$
    Since $\tau^{\geq 1} I\Omp$ is supported on $S$, the exact triangle above is simply given by 
    $$\bT \bR \underline{\Gamma_{S}}(f_{*}\Omega^{p}_{\tX}) \ar[r] & \bR \underline{\Gamma_{S}}(I \Omp) \ar[r] & \tau^{\geq 1} I\Omp \ar[r, "+1"] & \hfill. \eT$$
    By \cite{Siu}, $\cH^{i}(\bR \underline{\Gamma_{S}}(f_{*}\Omega^{p}_{\tX}))$ is coherent for $i \leq \codim_{X}(S) -1.$ By the exact triangle above, $\cH^{i}(\bR \underline{\Gamma_{S}}(I\Omp))$ must be coherent for $i \leq \codim_{X}(S) -1.$ Now the lemma follows from the assumption $k+1 \leq \codim_{X}(S) -1.$
\end{proof}

\begin{thm}\label{depthThm}
    Assume $k \leq \codim_{X}(S) -2$. If $I\Omp \simeq \cH^{0}(I\Omp)$ for $p \leq k$, then
    $$\depth_{S}(I\underline{\Omega}^{p}_{X}) \geq k +2 \quad \text{for $0 \leq p \leq n$}.$$
     
\end{thm}

\begin{proof}
    We follow the proof of \cite[Thm 4.1]{AH2}.  We need to show $\cH^{i}(\bR \underline{\Gamma_{S}}(I\Omp)) = 0$ for $i \leq k+1.$ For any closed point $x \in S$, we  have a spectral sequence
    $$E^{\alpha,\beta}_{2} = \cH^{\alpha}(\bR \underline{\Gamma_{x}}(\cH^{\beta}(\bR \underline{\Gamma_{S}}(I\underline{\Omega}^{p}_{X})))) \Rightarrow \cH^{\alpha +\beta}(\bR \underline{\Gamma_{x}}\bR \underline{\Gamma_{S}}(I\underline{\Omega}^{p}_{X})) \simeq \cH^{\alpha+\beta}(\bR \underline{\Gamma_{x}}(I\underline{\Omega}^{p}_{X})).$$ 
    If the dimension of the support of $\cH^{i}_{S}(I\underline{\Omega}^{p}_{X}):=\cH^{i}(\bR \underline{\Gamma_{S}}(\underline{\Omega}^{p}_{X}))$ is zero for $i \leq k+1$, then $\cH^{i}_{S}(I\underline{\Omega}^{p}_{X}) = 0$ for $i \leq k+1$ by Proposition \ref{depthProp}. So, we need to show that the dimension of the support of $\cH^{i}_{S}(I\underline{\Omega}^{p}_{X})$ is zero for $i \leq k+1.$ 
        
        The proof is by induction on the dimension of $X$. If $\dim X = 2$, then $X$ has isolated singularities, and the theorem holds from Proposition \ref{depthProp}. If $H$ is a sufficiently general hyperplane section, then $H$ is normal and, by \cite[Prop. 4.18]{ks}, we have exact sequence
        $$\bT0 \arrow[r]& I\underline{\Omega}^{p-1}_{H} \otimes \cO_{H}(-H) \arrow[r]   & I\underline{\Omega}^{p}_{X} \otimes \cO_{H} \arrow[r]   & I\underline{\Omega}^{p}_{H} \arrow[r] & 0, \eT$$
        with the long exact sequence
        $$\cdots \rightarrow \cH^{i}(I\underline{\Omega}^{p-1}_{H}) \otimes \cO_{H}(-H) \rightarrow \cH^{i}(I\underline{\Omega}^{p}_{X}) \otimes \cO_{H} \rightarrow \cH^{i}( I\underline{\Omega}^{p}_{H}) \rightarrow \cH^{i+1}(I\underline{\Omega}^{p-1}_{H}) \otimes  \cO_{H}(-H) \rightarrow  \cdots.$$
    Therefore, we have $\cH^{i}(I\underline{\Omega}^{p}_{H}) = 0$ for $i \geq 1$ and $p \leq k$. Furthermore, we may choose $H$ such that the singular locus of $H$ is $S \cap H$. So, the codimension of $(S \cap H)$ in $H$ is greater than or equal to $k+2$.

    Let $U = X \backslash S$ with natural map $j: U \hookrightarrow X$, and $U_{H} = U \cap H = H \backslash (H \cap S)$ with natural map $\varrho: U_{H} \hookrightarrow H.$ We have the commutative diagram with the rows and columns exact,
        $$\bT
        \hfill & \hfill & \hfill\\
   \bR \varrho_{*}I\underline{\Omega}^{p-1}_{U_H} \otimes \cO_{H}(-H) \arrow[r] \arrow[u, "+1"]  & \bR j_{*}I\underline{\Omega}^{p}_{U} \otimes^{L} \cO_{H} \arrow[r] \arrow[u, "+1"] & \bR \varrho_{*}I\underline{\Omega}^{p}_{U_H} \arrow[r, "+1"]  \arrow[u, "+1"] & \hfill \\
   I\underline{\Omega}^{p-1}_{H} \otimes \cO_{H}(-H) \arrow[r] \arrow[u]  & I\underline{\Omega}^{p}_{X} \otimes^{L} \cO_{H} \arrow[r]  \arrow[u] & I\underline{\Omega}^{p}_{H} \arrow[r, "+1"] \arrow[u] & \hfill.\\
    \bR \underline{\Gamma_{S \cap H}}(I\underline{\Omega}^{p-1}_{H}) \otimes \cO_{H}(-H) \arrow[r] \arrow[u]  & \bR \underline{\Gamma_{S }}(I\underline{\Omega}^{p}_{X})\otimes^{L} \cO_{H} \arrow[r]\arrow[u]  & \bR \underline{\Gamma_{S \cap H}}(I\underline{\Omega}^{p}_{H}) \arrow[r, "+1"] \arrow[u]\ & \hfill.
    \eT$$
    We have the long exact sequence of local cohomology
    $$ \cdots \rightarrow \cH^{i}(  \bR  \underline{\Gamma_{S \cap H}}(I\underline{\Omega}^{p-1}_{H})) \otimes \cO_{H}(-H) \rightarrow \cH^{i}(\bR \underline{\Gamma_{S }}(I\underline{\Omega}^{p}_{X})\otimes^{L} \cO_{H}) \rightarrow \cH^{i}(\bR \underline{\Gamma_{S \cap H}}(I\underline{\Omega}^{p}_{H})) \rightarrow \cdots. $$
    By the induction hypothesis, $\cH^{i}(  \bR  \underline{\Gamma_{S \cap H}}(I\underline{\Omega}^{p-1}_{H})) \otimes \cO_{H}(-H) = \cH^{i}(\bR \underline{\Gamma_{S \cap H}}(I\underline{\Omega}^{p}_{H}))= 0$ for $i \leq k +1$. Thus, 
    $$ \cH^{i}(\bR \underline{\Gamma_{S }}(I\underline{\Omega}^{p}_{X})\otimes^{L} \cO_{H})  = 0 \quad \text{for $i \leq k+1.$}$$
    The natural map
    $$ \cH^{i}(\bR \underline{\Gamma_{S}}(I\underline{\Omega}^{p}_{X})) \otimes \cO_{X}(-H) \rightarrow  \cH^{i}(\bR \underline{\Gamma_{S}}(I\underline{\Omega}^{p}_{X}))$$
    must be surjective for $i \leq k+1$. Let $A_{i} = \supp{\cH^{i}(\bR \underline{\Gamma_{S }}(I\underline{\Omega}^{p}_{X})})$ and suppose that $\dim A_{i} \geq 1$ for $i \leq k+1.$ For each $i \leq k +1$, we may choose $H$ such that $A_{i} \cap H \neq \emptyset$. If we localize at any closed point $x \in A_{i} \cap H$, there is a surjective map
    $$ \cH^{i}(\bR \underline{\Gamma_{S}}(I\underline{\Omega}^{p}_{X}))_x \otimes \cO_{X}(-H)_x \rightarrow  \cH^{i}(\bR \underline{\Gamma_{S}}(I\underline{\Omega}^{p}_{X}))_x.$$
    Since $\cH^{i}(\bR \underline{\Gamma_{S}}(I\underline{\Omega}^{p}_{X}))_{x}$ is a coherent for $i \leq k+1$ by the previous lemma, we can conclude that $\cH^{i}(\bR \underline{\Gamma_{S}}(I\underline{\Omega}^{p}_{X}))_x = 0$ for $i \leq k+1$ by applying Nakayama's lemma. Contradicting that $x \in A_{i}$, hence $\dim A_{i} =0$ for $i \leq k +1$.
\end{proof}

\begin{thm}\label{depththm3}
   Assume $k \leq \codim_{X}(S)-2$. If  $X$ has pre-$k$-rational singularities, then
    $$\depth_{S}(\underline{\Omega}^{p}_{X}) \geq k+2  \quad \text{for $0 \leq p \leq  \dim(X)$}.$$ 
\end{thm}

\begin{proof}
    To show $\depth_{S}(\Omp) \geq k+2$ for $k \leq \codim_{X}(S)-2$, it suffices to have $\depth(\underline{\Omega}^{p}_{X}) \geq k+2.$ Indeed, for any closed point $x \in S$, we  have a spectral sequence
    $$E^{\alpha,\beta}_{2} = \cH^{\alpha}(\bR \underline{\Gamma_{x}}(\cH^{\beta}(\bR \underline{\Gamma_{S}}(\underline{\Omega}^{p}_{X})))) \Rightarrow \cH^{\alpha +\beta}(\bR \underline{\Gamma_{x}}\bR \underline{\Gamma_{S}}(\underline{\Omega}^{p}_{X})) \simeq \cH^{\alpha+\beta}(\bR \underline{\Gamma_{x}}(\underline{\Omega}^{p}_{X})).$$
    If $\depth(\underline{\Omega}^{p}_{X}) \geq k+2,$ then to prove  $\depth_{S}(\underline{\Omega}^{p}_{X}) \geq k+2$, it suffices to show dimension of the support of $\cH^{i}_{S}(\underline{\Omega}^{p}_{X}):=\cH^{i}(\bR \underline{\Gamma_{S}}(\underline{\Omega}^{p}_{X}))$ is zero for $i \leq k+1$.  We may apply the same proof as in Theorem \ref{depthThm} to prove $\depth_{S}(\underline{\Omega}^{p}_{X}) \geq k+2$ because we may consider a sufficiently general hyperplane section $H$ by \cite[\S V.2.2.1]{GNPP} and apply induction. So, the theorem follows from Theorem \ref{depthThm2}.
\end{proof}

\begin{cor}
    Assume $k \leq codim_{X}(S)-2$. Let $U = X \backslash S$ with inculsion map $j: U \hookrightarrow X.$ If $X$ has pre-$k$-rational singularities, then for $0 \leq p \leq n,$ the natural maps
    $$\cH^{i}(\underline{\Omega}_{X}^{p}) \rightarrow \cH^{i}(I\underline{\Omega}^{p}_{X}) \rightarrow R^{i}j_{*}\Omega^{p}_{U}$$
    are an isomorphisms for $i \leq k$ and injective for $i = k+1.$
\end{cor}

\begin{proof}
    By Theorem \ref{depthThm2}, we have the natural map
    $$\cH^{i}(\underline{\Omega}_{X}^{p}) \rightarrow \cH^{i}(I\underline{\Omega}^{p}_{X})$$
    to be an isomorphism for $i \leq k$ and injective for $i = k+1$. By Theorem
    \ref{depthThm} and Theorem \ref{depththm3}, $\cH^{i}_{S}(I\Omp) = \cH^{i}_{S}(\Omp) = 0$ for $i \leq k+1$. Now use the exact triangles
    $$\bT \bR \underline{\Gamma_{S}}(I\Omp) \arrow[r] & I\Omp \arrow[r] & \bR j_{*}\Omega^{p}_{U} \arrow[r,"+1"] & \hfill \eT$$
     $$\bT \bR \underline{\Gamma_{S}}(\Omp) \arrow[r] & \Omp \arrow[r] & \bR j_{*}\Omega^{p}_{U} \arrow[r,"+1"] & \hfill. \eT$$
    
    %By \ref{phimap}, we must have $\cH^{i}_{S}(\Omega^{p}(\mathcal{K})) = 0$ for $i \leq k.$ Since $\Omega^{p}(\mathcal{K})$ is supported on $S$, we must have $\cH^{i}(\Omega^{p}(\mathcal{K})) = 0$ for $i \leq k.$
\end{proof}

\begin{cor}
     If $X$ has pre-$k$-rational singularities, then 
    $$\cH^{i}_{S}(\cH^{0}(\Omp)) = 0 \quad \text{for $0 \leq i \leq \min \{\codim_{X}(S)-2,k\}+1,$ for $0 \leq p \leq k,$}$$
    
    $$\cH^{i}_{S}(\Omega^{p}_{X}) = 0 \quad \text{for $2 \leq i \leq \min \{\codim_{X}(S)-2,k\}+1,$ for $0 \leq p \leq k.$}$$
\end{cor}

\begin{proof}
     Let $U = X \backslash S$ with inclusion map $j: U \hookrightarrow X.$ By the previous corollary, if $X$ has pre-$k$-rational singularities, then for $0 \leq p \leq k,$ we have
    $$\cH^{0}(\Omp) \cong j_{*}\Omega^{p}_{U},$$

    $$ 0=\cH^{i}(\underline{\Omega}_{X}^{p}) \cong R^{i}j_{*}\Omega^{p}_{U} \cong \cH^{i+1}_{S}(\Omega^{p}_{X}) \cong \cH^{i+1}_{S}(\cH^{0}(\Omp)) \quad \text{for $1\leq i \leq \min \{\codim_{X}(S)-2,k\}.$}$$
\end{proof}

\section{On the local cohomological dimension}

Again, we will keep the same assumptions as the first section. Consider a smooth variety $Y$ of dimension $m$ and assume $X \subseteq Y$ is a closed subvariety with inclusion map $i: X \hookrightarrow Y$. We define the local cohomological dimension of $X$ in $Y$ to be
$$\lcd(Y,X) = \max\{q| \hspace{.05in}\cH^{q}_{X}(\cO_{Y}) \neq 0\}.$$
For each $q \in \Z$, it is well known $\cH^{q}_{X}(\cO_{Y})$ carries a left $\cD_{Y}$-module structure. In fact, it is the underlying $\cD_{Y}$-module for the mixed Hodge module $\cH^{q-d}\bigg(i_{+}\D(\Q^{H}_{X}[n]) \bigg)$, where $\D(\Q^{H}_{X}[n])$ is the dual of $\Q^{H}_{X}[n]$ in the derived category of mixed Hodge modules, and $d = m-n$. It was shown by Musta\c{t}\u{a} and Popa \cite{MuPo} that the local cohomological dimension of $X$ in $Y$ can be described by using the Hodge filtration of $\cH^{q}\bigg(i_{+}\D(\Q^{H}_{X}[n]) \bigg)$ for $q \in \Z.$ The key observations are
$$\cH^{q}_{X}(\cO_{Y}) = 0 \quad \text{if and only if} \quad \cH^{q-d}\bigg(i_{+}\D(\Q^{H}_{X}[n]) \bigg)= 0.$$
$$\cH^{q-d}\bigg(i_{+}\D(\Q^{H}_{X}[n]) \bigg)= 0 \quad \text{if and only if} \quad \cH^{0}\bigg(Gr^{F}_{p}DR\bigg(\cH^{q-d}\bigg(i_{+}\D(\Q^{H}_{X}[n]) \bigg)\bigg) \bigg)= 0 \quad \text{$p \geq 0$.}$$
There is a spectral sequence
$$E^{r,q-d}_{2}=\cH^{r}\bigg( Gr^{F}_{p}DR\bigg(\cH^{q-d}\bigg(i_{+}\D(\Q^{H}_{X}[n]) \bigg)\bigg) \bigg) \Rightarrow \cH^{r+q-d}\bigg(Gr^{F}_{p}DR(i_{+}\D(\Q^{H}_{X}[n])) \bigg),$$
and a quasi-isomorphism
$$\cH^{r+q-d}\bigg(Gr^{F}_{p}DR(i_{+}\D(\Q^{H}_{X}[n])) \bigg) \simeq \cH^{p+r+q}(\bR \cH om_{\cO_{Y}}(i_{*}\underline{\Omega}^{p}_{X}, \omega_{Y})).$$
Using these facts, Musta\c{t}\u{a} and Popa prove the following theorem to obtain an upper bound for $\lcd(Y, X)$.
\begin{thm}\cite{MuPo}\label{MPthm}
   For every positive integer $c$, the following are equivalent:
   \begin{enumerate}
       \item $\lcd(Y,X) \leq c$
       \item $\cE xt^{\ell+q+1}_{\cO_{Y}}(i_{*}\underline{\Omega}^{q}_{X}, \omega_{Y}) = 0$ for all $\ell\geq c$ and $q \geq 0.$
   \end{enumerate}
\end{thm}
Assume $k \leq n-2$. If $X$ has pre-$k$-rational singularities, the maps
$$\cH^{0}(\Omp) \rightarrow \Omp \rightarrow I \Omp \rightarrow \D(\underline{\Omega}^{n-p}_{X})[-n]$$
are a quasi-isomorphisms for $p \leq k$. If we dualize, we obtain quasi-isomorphisms
$$\underline{\Omega}^{n-p}_{X} \simeq I\underline{\Omega}^{n-p}_{X} \simeq \D(\underline{\Omega}^{p}_{X})[-n] \quad \text{for $0 \leq p \leq k.$}$$
Thus,
$$\depth(\Omp) \geq n-p \quad \text{for $0 \leq p \leq k$.}$$
$$\depth(\Omp) =n \quad \text{for $n-k \leq p$.}$$
Also, by Theorem \ref{depthThm2}, we have
$$\depth(\underline{\Omega}^{p}_{X}) \geq k +2 \quad \text{for $0 \leq p.$}$$
By Remark \ref{RHMrmk}, it suffices to consider when $n \geq 2k + 3$. So, the optimal bounds for the depth of the graded pieces of the Du Bois complex are given by
$$\depth(\Omp) \geq n-p \quad \text{for $0 \leq p \leq k$}$$
$$\depth(\underline{\Omega}^{p}_{X}) \geq k +2 \quad \text{for $k+1 \leq p \leq n-k -1$}$$
$$\depth(\underline{\Omega}^{p}_{X}) =n \quad \text{for $ n-k \leq p.$}$$
%Thus, if $X$ has pre-$k$-rational singularities and  $n \geq 2k +3$, then
%$$\cE xt^{m-i}_{Y}(i_{*}\underline{\Omega}^{q}_{X}, \omega_{Y}) = 0 \quad \text{for $ i \leq n-k-1$ and $0 \leq q \leq k$.}$$
%$$\cE xt^{m-i}_{Y}(i_{*}\underline{\Omega}^{q}_{X}, \omega_{Y}) = 0 \quad \text{for $ i \leq k+1$ and $k+1\leq q \leq n-k-1$}$$
%$$\cE xt^{m-i}_{Y}(i_{*}\underline{\Omega}^{q}_{X}, \omega_{Y}) = 0 \quad \text{for $ i \leq n-1$ and $n-k \leq q$}$$
Therefore, we have $2k +3 \leq   \min\{\depth(\underline{\Omega}^{q}_{X}) +q \}$ and we have the equation
$$\cE xt^{m-i+1}_{\cO_{Y}}(i_{*}\underline{\Omega}^{q}_{X}, \omega_{Y}) = 0 \quad \text{for} \quad i +q \leq \max
\begin{cases}
    2k+3 \\ \\
    \min\{\depth(\underline{\Omega}^{q}_{X}) +q \} 
\end{cases}   
0 \leq q.$$
If we set $\ell = m -i -q,$ then 
%$$ \cE xt^{m-i}_{Y}(i_{*}\underline{\Omega}^{q}_{X}, \omega_{Y}) = 0 \quad \text{for} \quad 0 \leq i \leq k+1 \quad \text{and} \quad 0 \leq q .$$
%$$\cE xt^{\ell+q +1}_{Y}(i_{*}\underline{\Omega}^{q}_{X}, \omega_{Y}) = 0 \quad \text{for$ \quad 0\leq (m-\ell -q -1) +q \leq \min\{\depth(\underline{\Omega}^{q}_{X}) +q \}$ and $k+1 \leq q \leq n-k-1$.}$$
$$\cE xt^{\ell +q +1}_{\cO_{Y}}(i_{*}\underline{\Omega}^{q}_{X}, \omega_{Y}) = 0 \quad \text{for} \quad \ell \geq m- \max\begin{cases}
    2k+3 \\ \\
    \min\{\depth(\underline{\Omega}^{q}_{X}) +q \} 
    \end{cases}
    \quad \text{$0 \leq q$.}$$
%In particular, we always have
%$$\cE xt^{r +q +1}_{Y}(i_{*}\underline{\Omega}^{q}_{X}, \omega_{Y}) = 0 \quad \text{for} \quad m-2k-3 \leq r \quad \text{and} \quad 0 \leq q.$$
Thus, if $X$ has pre-$k$-rational singularities, then we always have 
$$\lcd(Y,X) \leq \max\{m-2k-3, \codim_{Y}(X)\}.$$
The local cohomological defect $\lcdef(X)$ of $X$ is defined as
$$\lcdef(X) := \lcd(Y,X) - \codim_{X}(Y).$$
 See \cite{PopaShen}, \cite{ParkPopa}, or \cite{shenVenVo} for more details on the local cohomological defect. With the upper bound $\lcd(Y, X) \leq m-2k-3$ for varieties with pre-$k$-rational singularities, we obtain an upper bound for the local cohomological defect of $X$. 

\begin{cor} \cite{ParkPopa}
    If $X$ has pre-$k$-rational singularities, then
    $$\lcdef(X) \leq \max \{ n-2k-3, 0 \}.$$
\end{cor}

\begin{rmk}
    The inequality above holds as long as $X$ satisfies the $(*)_{k}$ (or $HRH(X) \geq k$) condition  \cite[Corollary 7.5]{ParkPopa}. 
\end{rmk}

\begin{cor}
    If $X$ has pre-$(k-1)$-rational singularities and the map $$\cH^{i}(\underline{\Omega}^{k}_{X}) \rightarrow \cH^{i}(I\underline{\Omega}^{k}_{X})$$
    is an isomorphism for $ k<i$ and surjective for $i = k,$ then 
     $$\lcdef(X) \leq \max \{ n-2k-3, 0 \}.$$
\end{cor}

\begin{proof}
    This follows from Corollary \ref{*k+1} and \cite[Corollary 7.5]{ParkPopa}.
\end{proof}

\printbibliography
\end{document}